\documentclass[A4paper,reqno,oneside,11pt]{amsart}

\usepackage[foot]{amsaddr}
\usepackage{fullpage}
\usepackage{setspace}
\usepackage{amsmath,amssymb,amsfonts,amsthm,mathtools,mathrsfs}
\usepackage{tikz}

\usetikzlibrary{arrows.meta,positioning,calc}
\usepackage{enumitem,graphicx}
\usepackage{xcolor}
\usepackage[backref=page]{hyperref}
\usepackage[capitalize,nameinlink]{cleveref}
\usepackage{thmtools}

\usepackage{orcidlink}

\hypersetup{
  colorlinks,
  linkcolor={red!50!black},
  citecolor={green!50!black},
  urlcolor={blue!50!black}
}

\definecolor{Accent}{RGB}{34,74,126}
\definecolor{OrbitBlue}{RGB}{31,119,180}
\definecolor{BridgeRed}{RGB}{180,50,47}
\definecolor{Soft}{RGB}{243,247,251}

\allowdisplaybreaks

\newtheorem{theorem}{Theorem}[section]
\RenewCommandCopy{\theHtheorem}{\thetheorem}
\newtheorem{lemma}[theorem]{Lemma}
\RenewCommandCopy{\theHlemma}{\thelemma}

\RenewCommandCopy{\theHclaim}{\theclaim}

\RenewCommandCopy{\theHproposition}{\theproposition}
\newtheorem{fact}[theorem]{Fact}
\RenewCommandCopy{\theHfact}{\thefact}

\RenewCommandCopy{\theHobservation}{\theobservation}
\newtheorem{corollary}[theorem]{Corollary}
\RenewCommandCopy{\theHcorollary}{\thecorollary}

\RenewCommandCopy{\theHconjecture}{\theconjecture}
\newtheorem{question}[theorem]{Question}
\RenewCommandCopy{\theHquestion}{\thequestion}

\theoremstyle{definition}

\newtheorem*{remark*}{Remark}
\newtheorem{definition}[theorem]{Definition}

\newcommand{\cO}{\mathcal{O}}
\newcommand{\cT}{\mathcal{T}}

\newcommand{\N}{\mathbb{N}}

\newcommand{\R}{\mathbb{R}}

\newcommand{\abs}[1]{\left| #1 \right|}
\newcommand{\of}[1]{\left( #1 \right)}
\newcommand{\set}[1]{\left\{ #1 \right\}}

\DeclareMathOperator{\KG}{KG}
\DeclareMathOperator{\Cay}{Cay}
\DeclareMathOperator{\conv}{conv}
\DeclareMathOperator{\Flip}{Flip}

\newcommand{\Assoc}{\mathrm{Assoc}}
\newcommand{\Diag}{\mathrm{Diag}}
\newcommand{\GrayC}[1]{\mathrm{HC}^{\mathrm{ear}}_{#1}}
\newcommand{\Perm}{\mathrm{Perm}}

\newcommand{\comment}[1]{}

\begin{document}

\setstretch{1.27}

\title{Kneser graphs of triangulations are Hamiltonian}

\author[Anton Molnar]{Anton Molnar\:\orcidlink{0009-0003-2824-8382}}
\author[Cosmin Pohoata]{Cosmin Pohoata\:\orcidlink{0009-0004-4842-2939}}
\author[Michael Zheng]{Michael Zheng\:\orcidlink{0009-0001-2406-6130}}
\address{Department of Mathematics, Emory University, Atlanta, GA, 30322, USA}
\email{\{anton.molnar, cosmin.pohoata, xiangxiang.michael.zheng\}@emory.edu}

\maketitle

\medskip
\begin{abstract}
For every $n \geq 5$, we show that the Kneser graph of triangulations of a convex $n$-gon contains a Hamiltonian cycle.
\end{abstract}

\section{Introduction} \label{sec:intro}
For a finite set system $\mathcal{F}$, the Kneser graph $\KG(\mathcal{F})$ is the graph whose vertices are the members of $\mathcal{F}$ and two vertices $F$ and $F'$ are connected by an edge if $F\cap F' = \emptyset$. When $\mathcal{F} = \binom{[n]}{k}$ and $n \geq 2k$, the resulting Kneser graph encodes two of the most celebrated results in extremal combinatorics. On the one hand, the independence number of this graph is determined by the Erd\H{o}s-Ko-Rado theorem \cite{EKR1961}, stating that an intersecting family has size at most $\binom{n-1}{k-1}$. On the other hand, in a highly influential paper \cite{Lovasz1978}, Lov\'asz proved that the chromatic number of $\smash{\KG\!\left(\binom{[n]}{k}\right)}$ is $n-2k+2$, the first application of topological methods in combinatorics.

Hamiltonicity has also emerged as a parallel and unexpectedly rich theme in the study of $\smash{\KG\of{\binom{[n]}{k}}}$. Since the Kneser graph $\smash{\KG\of{\binom{[n]}{k}}}$ is connected and vertex-transitive for $n\ge 2k+1$, Hamiltonicity may be viewed as a concrete instance of the Lov\'asz problem on Hamiltonian cycles in connected vertex-transitive graphs. See for example \cite{Lovasz1970}, \cite{ChristofidesHladkyMathe2014} and the references therein. 
Until recently, results about Hamiltonicity of $\KG\of{\binom{[n]}{k}}$ have been developed for either sufficiently dense Kneser graphs or very sparse Kneser graphs. It was first proved by Heinrich and Wallis \cite{HeinrichWallis1978} for $n\geq (1+o(1))k^2/\ln{2}$ and was subsequently improved by B. Chen and Lih \cite{ChenLih1987} for $n\geq (1+o(1))k^2/\log{k}$. The first linear bound on $n$ in terms of $k$ was proved by Y. Chen \cite{Chen2000} for $n\ge 3k$ and later improved to $n\ge 2.62k+1$ \cite{Chen2000,Chen2003}. Moreover, Y. Chen and F\"uredi gave a short proof in the divisible case $k\mid n$ with $n \geq 3k$ \cite{ChenFuredi2002}. Removing the divisibility condition of \cite{ChenFuredi2002}, Bellmann and Sch\"ulke gave a simple and elegant proof for $n\ge 4k$ \cite{BellmannSchuelke2021}. In the sparse case, it was shown by M\"utze, Nummenpalo, and Walczak \cite{MutzeNummenpaloWalczak2021} that $\KG\of{\binom{[2k+1]}{k}}$ is Hamiltonian. Finally, in \cite{MerinoMutzeNamrata2023}, Merino, M\"utze, and Namrata established that $\smash{\KG\of{\binom{[n]}{k}}}$ is Hamiltonian for all $n \geq 2k+1$, except for the case of the Petersen graph $\KG\of{\binom{[5]}{2}}$.

In this paper, we study a natural geometric analogue of this problem. Let $\mathcal{P}$ be a fixed convex $n$-gon with vertex set $[n]$ (in cyclic order), and let $\cT_n$ be the set of triangulations of $\mathcal{P}$.
We identify each triangulation as the set of its $n-3$ pairwise noncrossing diagonals, where
\[
\Diag(n)=\bigg\{\{i,j\}\in \binom{[n]}{2}: i-j\not\equiv \pm 1 \pmod n\bigg\}.
\]
The Kneser graph of triangulations, $\KG(\cT_n)$, is the graph whose vertex set is $\cT_n$ and where two triangulations are adjacent if and only if they share no diagonal.
It is a classical result that $\abs{\cT_n}=C_{n-2}$, where $C_m=\frac{1}{m+1}\binom{2m}{m}$ is the $m$th Catalan number. 

Our main result is the following.
\begin{theorem}\label{thm:main}
For every $n\ge 5$, the graph $\KG(\cT_n)$ has a Hamiltonian cycle.
\end{theorem}

Theorem \ref{thm:main} continues the line of work from \cite{MolnarPohoataZhengZhu2025}, where Zhu and the authors recently established a Lov\'asz-Kneser theorem for $\KG(\cT_n)$. We refer the interested reader to the introduction in \cite{MolnarPohoataZhengZhu2025} for more context. With this connection in mind, it is natural to also phrase the Hamiltonicity problem in polyhedral language. For a polytope $\mathcal{Q}$, define its Kneser graph $\KG(\mathcal{Q})$ to be the graph whose vertices are the vertices of $\mathcal{Q}$, with two vertices adjacent if and only if they lie on no common facet. The associahedron $\Assoc_{n-3}$ is an $(n-3)$-dimensional polytope whose vertices correspond to triangulations of an $n$-gon and whose $1$-skeleton is $\Flip(n)$. Since the facets of $\Assoc_{n-3}$ correspond to diagonals of the $n$-gon, we have
\begin{equation} \label{assoc}
\KG(\cT_n)=\KG(\Assoc_{n-3}).
\end{equation}
This perspective played a fundamental role in the proof that $\chi(\KG(\cT_n)) = n-2$ from \cite{MolnarPohoataZhengZhu2025}, which combined topological tools with geometric properties of the associahedron. In contrast, our proof of Theorem \ref{thm:main} here will be purely combinatorial. However, the relation from \eqref{assoc} comes with an interesting geometric consequence, as well as a few natural further questions. We first record the following immediate corollary of Theorem \ref{thm:main}.
\begin{corollary}
    For $n \geq 5$, there exists an enumeration $v_1, \dots, v_{C_{n-2}}$ of the vertices of $\Assoc_{n-3}$ such that the open segment $v_{i}v_{i+1}$ lies within the interior of $\Assoc_{n-3}$, for every $i=1,\ldots,C_{n-2}$. 
\label{thm:main_polytope}
\end{corollary}

In particular, \Cref{thm:main_polytope} suggests that Hamiltonicity phenomena for Kneser graphs may be viewed as a kind of ``hidden cyclic visibility'' property of the vertex set of a polytope. 
Another polytope whose Kneser graph can be studied in the same spirit is the permutohedron
\[
\Perm_n \coloneqq  \conv\set{(\sigma(1),\dots,\sigma(n)) : \sigma\in S_n}\subset \R^n.
\]
Its vertex set is naturally identified with $S_n$, and a standard facet description is
\[
\sum_{i\in I}x_i \ge \binom{|I|+1}{2}
\qquad (\emptyset\neq I\subsetneq [n]),
\]
together with the affine equation $\sum_{i=1}^n x_i=\binom{n+1}{2}$; see \cite{Postnikov2009}. In \Cref{sec:perm}, we will use the ideas from Theorem \ref{thm:main} to also show that $\KG(\Perm_n)$ is Hamiltonian, where the analysis is significantly simpler. This is due to the fact that $\KG(\Perm_n)$ is a much denser graph than $\KG(\Assoc_{n-3})$. Curiously, it turns out that $\KG(\Perm_n)$ is so dense that it contains a $k$th power of a Hamiltonian cycle, for every $k \geq 1$. 

\begin{theorem}\label{thm:permutohedron}
Let $k\geq 1$. Then for $n$ sufficiently large, $\KG(\Perm_n)$ contains a $k$th power of a Hamiltonian cycle. 
\end{theorem}

In striking contrast with \Cref{thm:permutohedron}, the Kneser graph of the associahedron does not admit higher powers of Hamiltonian cycles in general. Indeed, the obstruction already appears for the square of a Hamiltonian cycle. If a graph contains the square of a Hamiltonian cycle, then every vertex lies in a triangle, since each vertex is adjacent to its two immediate neighbors and to the two vertices at distance two along the cycle. However, $\KG(\Assoc_{n-3})=\KG(\cT_n)$ has vertices that lie in no triangle at all; for example, this happens for the star triangulations (that is, triangulations whose diagonals all share a common endpoint). See for example Proposition $5.2$ in \cite{MolnarPohoataZhengZhu2025}. Consequently, $\KG(\Assoc_{n-3})$ cannot contain the square of a Hamiltonian cycle, and hence cannot contain the $k$th power of a Hamiltonian cycle for any $k\ge 2$. 

Before we proceed to the proofs of Theorem \ref{thm:main} and Theorem \ref{thm:permutohedron}, let us also say a few words about the high-level idea that we will be pursuing.

\subsection*{Proof overview}

In a few words, our proof may be viewed as a geometric analogue of the strategy used by Merino, M\"utze, and Namrata in \cite{MerinoMutzeNamrata2023} (which itself builds upon previous ideas by Chen and F\"uredi \cite{ChenFuredi2002}), though the execution in our setting is considerably simpler.

The starting point is the \emph{flip graph} $\Flip(n)$ on triangulations of the $n$-gon: its vertex set is $\cT_n$, and two triangulations are adjacent if they differ by a single flip, that is, if one is obtained from the other by replacing one diagonal with the other diagonal of the corresponding quadrilateral. The graph $\Flip(n)$ is the $1$-skeleton of the associahedron, and it is Hamiltonian by work of Lucas \cite{Lucas1987}; see also Hurtado and Noy \cite{HurtadoNoy1999}. On the other hand, the Kneser graph $\KG(\cT_n)$ admits a canonical cycle decomposition coming from rotations of the polygon: as shown in \cref{lem:orbit_cycle,lem:rotationorbits}, each rotation orbit induces a cycle in $\KG(\cT_n)$ (except for the small exceptional case $n=6$, which is handled separately). The idea then becomes to merge these orbit-cycles into a single Hamiltonian cycle. To do this, we use a sparse guide cycle inside the flip graph. Rather than working directly with a Hamiltonian cycle of $\Flip(n)$, we start from a Hamiltonian cycle of $\Flip(n-1)$ and append the diagonal $\{1,n-1\}$, which we refer to as an \emph{ear}. This produces a cycle $\GrayC{n}\subseteq \Flip(n)$ whose vertices all contain the same diagonal. In particular, $\GrayC{n}$ is an independent set in $\KG(\cT_n)$, while at the same time it meets every rotation orbit. The guide cycle therefore provides a controlled way of selecting representatives from the orbit-cycles.

The key local step is \cref{lem:flips}: if two triangulations differ by a flip, then after rotating one of them by $\pm1$, the two become disjoint. This allows us to convert transitions in the guide cycle into edges of $\KG(\cT_n)$ connecting different rotation orbits. Finally, a greedy algorithm shows that these connecting edges can be chosen compatibly, so that the orbit-cycles can be glued together one by one into a Hamiltonian cycle. It is perhaps worth contrasting this last part of the argument with the analogous part of the strategy from \cite{MerinoMutzeNamrata2023}. To build a Hamiltonian cycle in $\KG({[n] \choose k})$, Merino, M\"utze and Namrata also start from a structured decomposition together with an auxiliary Hamiltonian object, but there is no natural analogue of our geometric bridge lemma: passing from a local move in the auxiliary graph to an actual edge in the Kneser graph is substantially more delicate. A significant part of their argument is devoted precisely to the construction and control of such bridges. In contrast, in the triangulations setting, the ambient geometry provides the cycle decomposition as well as the bridges in a much more canonical way. 

\section{Proof of Theorem \ref{thm:main}}

We write $r$ for the clockwise rotation $i\mapsto i+1\pmod n$ of the polygon.

\begin{definition}
For a diagonal $d=\{i,j\}\in \Diag(n)$, define its rotation by
\[
r(d)\coloneqq \{i+1,j+1\},
\]
where indices are taken modulo $n$. For a triangulation $T\in \cT_n$, define
\[
r(T)\coloneqq \{r(d):d\in T\}.
\]
Two triangulations $T,T'$ are in the same \emph{rotation orbit}, written $T\approx T'$, if $T=r^k(T')$ for some $k\in \N$. Clearly, each rotation orbit is closed under $r$, i.e.
\end{definition}

\begin{fact}
If $T$ belongs to a rotation orbit $\cO$, then $r(T)\in \cO$.
\end{fact}

We next note that the only obstruction to a rotation orbit forming a genuine cycle is the possibility that it has size $2$.

\begin{lemma}\label{lem:orbit_cycle}
A rotation orbit has size $2$ if and only if $n\in\{4,6\}$ and every triangulation in the orbit contains only ears, i.e.\ diagonals of length $2$.
\end{lemma}

\begin{proof}
Suppose $\cO$ has size $2$, so $r^2(T)=T$ for every $T\in \cO$. If some $d\in T$ has length at least $3$, then $d$ and $r^2(d)$ cross, contradiction. Thus every diagonal in $T$ has length $2$. This can happen only for $n\in\{4,5,6\}$. For $n=5$, however, all triangulations lie in a single rotation orbit of size $5$. Hence only $n\in\{4,6\}$ is possible.

Conversely, for $n=4$ and $n=6$, the triangulations consisting only of ears indeed form orbits of size $2$.
\end{proof}

For $n \geq 7$, it thus follows from Lemma \ref{lem:orbit_cycle} that each rotation orbit must have size at least $3$. In contrast, the case $n=6$ is a bit special, and will require a slight modification of the overall strategy. We will address this at the end of this section.

\begin{lemma}\label{lem:rotationorbits}
If $\abs{\cO}\ge 3$ and $T\in \cO$, then
\[
T,r(T),\dots,r^{\abs{\cO}-1}(T),T
\]
forms a cycle in $\KG(\cT_n)$. In particular, for $n\ge 7$, each rotation orbit induces a cycle in $\KG(\cT_n)$.
\end{lemma}

\begin{proof}
It suffices to show that $T\cap r(T)=\emptyset$ for every $T\in \cT_n$. Suppose instead that $d\in T\cap r(T)$. Since $d\in T$, we also have $r(d)\in r(T)$. But $d$ and $r(d)$ cross, so they cannot both belong to the noncrossing triangulation $r(T)$.
\end{proof}

Thus, for $n\ge 7$, the rotation orbits partition the vertex set of $\KG(\cT_n)$ into cycles. Equivalently, they form a canonical $2$-factor
\[
F\subseteq \KG(\cT_n).
\]

\subsection{The bridge lemma and the guide cycle}

We now explain how to merge these rotation cycles. The key local observation is that if two triangulations differ by a flip, then after rotating one of them by $\pm1$, the two become disjoint.

\begin{lemma}\label{lem:flips}
Suppose $T,T'\in \cT_n$ differ by a flip. Then
\[
T\cap r(T')=\emptyset
\qquad\text{or}\qquad
T\cap r^{-1}(T')=\emptyset.
\]
\end{lemma}

\begin{proof}
Since $T$ and $T'$ differ by one flip, we may write
\[
T=(T'-\{d'\})\cup\{d\}
\]
for some diagonals $d,d'\in \Diag(n)$. Because $r(T')$ is disjoint from $T'-\{d'\}$, any intersection between $T$ and $r(T')$ must come from the new diagonal $d$. Hence
\[
T\cap r(T')\neq\emptyset \Longrightarrow d\in r(T'),
\]
so $d=r(d')$.

Similarly,
\[
T\cap r^{-1}(T')\neq\emptyset \Longrightarrow d=r^{-1}(d').
\]
Since $r(d')\neq r^{-1}(d')$, the two possibilities cannot both occur. Therefore at least one of $T\cap r(T')$ and $T\cap r^{-1}(T')$ is empty.
\end{proof}

To connect the rotation orbits, we use a sparse guide cycle in the flip graph. Let
\[
\mathrm{HC}_{n-1}=(T_1',\dots,T_N',T_1'),
\qquad N=C_{n-3},
\]
be a Hamiltonian cycle in the flip graph of $\cT_{n-1}$.
Appending the diagonal $\{1,n-1\}$ to every $T_i'$ yields triangulations
\[
T_i\coloneqq T_i'\cup\{\{1,n-1\}\}\in \cT_n.
\]
The resulting cycle
\[
\GrayC{n}\coloneqq (T_1,\dots,T_N,T_1)
\]
lies in $\Flip(n)$ and consists entirely of triangulations containing the fixed diagonal $\{1,n-1\}$.
In particular, $\GrayC{n}$ is an independent set in $\KG(\cT_n)$.
Moreover, every rotation orbit of $\cT_n$ contains a triangulation with an ear at the vertex $n$, so every rotation orbit is met by $\GrayC{n}$.

\subsection{The auxiliary tree on rotation orbits}

Define an auxiliary rooted tree $G$ on the rotation orbits as follows. Let $\cO_1$ be the rotation orbit of $T_1$, and take $\cO_1$ as the root. While traversing $\GrayC{n}$, add the edge $(\cO,\cO')$ whenever $\GrayC{n}$ moves from a triangulation in $\cO$ to a triangulation in $\cO'$, and this is the first occurrence of $\cO'$ along the traversal.

Since every rotation orbit is met by $\GrayC{n}$ and every new edge joins a previously seen orbit to a newly seen one, $G$ is a spanning tree on the set of rotation orbits.

The role of $G$ is to record which rotation cycles of the $2$-factor $F$ we will merge directly. The next lemma shows that no orbit is asked to participate in too many such mergers.

\begin{lemma}\label{lem:degree}
For every rotation orbit $\cO\in V(G)$,
\[
\deg_G(\cO)\le \abs{\cO}.
\]
\end{lemma}

\begin{proof}
Since $\GrayC{n}$ consists entirely of triangulations containing the fixed diagonal $\{1,n-1\}$, its vertex set is an independent set in $\KG(\cT_n)$. On the other hand, by \Cref{lem:rotationorbits}, every rotation orbit induces a cycle in $\KG(\cT_n)$. Therefore
\[
\abs{\GrayC{n}\cap \cO}\le \frac{\abs{\cO}}{2}.
\]

Now an edge involving $\cO$ in the construction of $G$ is created either
\begin{itemize}[leftmargin=1.5em]
\item when the first occurrence of $\cO$ is reached, or
\item when a representative of $\cO$ is followed in $\GrayC{n}$ by the first occurrence of some other orbit.
\end{itemize}
The former can happen at most once, while the latter can happen at most $\abs{\GrayC{n}\cap \cO}$ times. Thus
\[
\deg_G(\cO)\le 1+\frac{\abs{\cO}}{2}\le \abs{\cO},
\]
since $\abs{\cO}\ge 2$. 
\end{proof}

Let $\cO_1\prec \dots \prec \cO_m$ be the ordering of the rotation orbits according to their first appearance while traversing $\GrayC{n}$. For each $i\in [2,m]$, let $f(\cO_i)$ denote the unique parent of $\cO_i$ in $G$.

\subsection{A sparse enlargement of the $2$-factor}

We now build a sparse supergraph of the $2$-factor $F$ that records one edge between each orbit and its parent.

\begin{lemma}\label{lem:almost}
Let $F$ be the $2$-factor whose components are the rotation orbits. There exists a graph
\[
F\subseteq F'\subseteq \KG(\cT_n)
\]
such that
\begin{enumerate}[label=(\alph*),leftmargin=1.5em]
\item every edge of $F'\setminus F$ joins two distinct rotation orbits;
\item $\deg_{F'}(T)\le 3$ for every $T\in \cT_n$;
\item contracting every rotation orbit in $F'$ yields the tree $G$.
\end{enumerate}
\end{lemma}

\begin{proof}
It is enough to choose, for each $i\in [2,m]$, an edge
\[
e_i\in E_{\KG(\cT_n)}\bigl(f(\cO_i),\cO_i\bigr)
\]
such that if
\[
F_i\coloneqq F+\{e_2,\dots,e_i\},
\]
then $\deg_{F_i}(T)\le 3$ for every triangulation $T$.

We proceed by induction on $i\ge 1$. For $i=1$ there is nothing to prove, since $F$ is a $2$-factor. Assume now that the statement holds for $i-1\ge 1$.

Because $(f(\cO_i),\cO_i)\in E(G)$, there exist triangulations $T\in f(\cO_i),\qquad T'\in \cO_i$ such that $T$ and $T'$ differ by one flip. By \Cref{lem:flips}, there exists $\epsilon\in\{-1,1\}$ such that
\[
S\coloneqq r^\epsilon(T')\in \cO_i
\]
is disjoint from $T$. Hence $\set{T,S}$ is an edge of $\KG(\cT_n)$ between the two required rotation orbits.

Now rotate both endpoints simultaneously. For any integer $\sigma$, the edge $\set{r^\sigma(T), r^\sigma(S)}$ is again an edge between the same two rotation orbits. By \Cref{lem:degree}, the parent orbit $f(\cO_i)$ is incident in $G$ with at most $\abs{f(\cO_i)}$ edges, so among the edges $e_2,\dots,e_{i-1}$ there are at most $\abs{f(\cO_i)}-1$ already incident with representatives of $f(\cO_i)$. Therefore, by choosing $\sigma$ appropriately, we may arrange that adding $e_i\coloneqq \set{r^\sigma(T), r^\sigma(S)}$ does not increase the degree of any vertex in $f(\cO_i)$ above $3$. Since $\cO_i$ is being encountered for the first time, no previously chosen edge is incident with a representative of $\cO_i$, so the same holds for $\cO_i$. This completes the induction.
\end{proof}

\begin{figure}[htb!]
    \centering
    \includegraphics[width=0.7\linewidth]{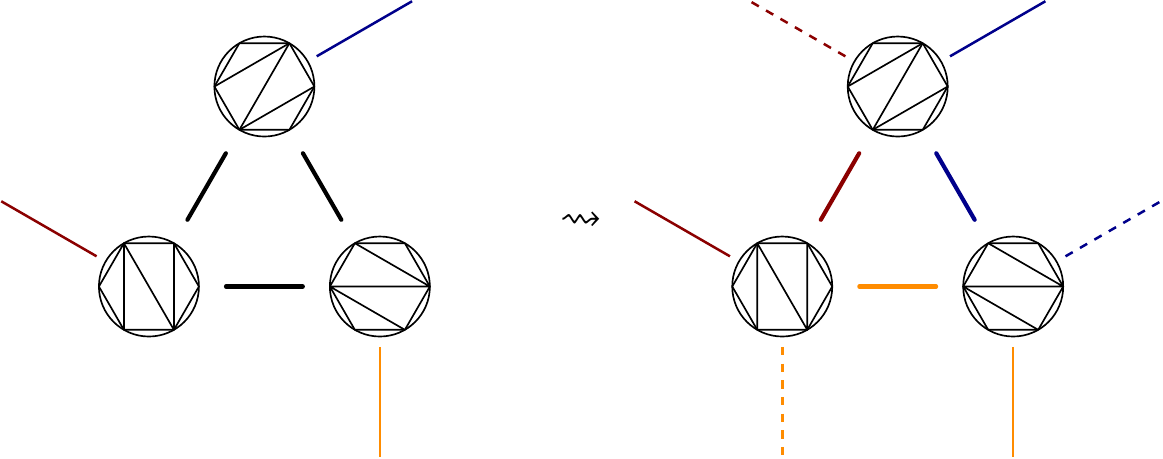}
    \caption{Left: A rotation orbit $\cO$ with $\abs{\cO} \geq 3$, with three bridges to other orbits. Right: the ``rotated'' bridge-copies dashed and the orbit edge in the color of the edge occupying it.}
    \label{fig:good-case}
\end{figure}
\subsection{The final splicing argument}

We are now ready to merge the rotation cycles one by one.

\begin{proof}[Proof of \Cref{thm:main}]
We have already noted that the theorem holds for $n=5$, and we will postpone the case $n=6$ to the end. Assume therefore that $n\ge 7$.

Let
\[
F'=F+\{e_2,\dots,e_m\}
\]
be a graph as in \Cref{lem:almost}. For $k=2,\dots,m$, write
\[
e_k=\{S,T\},
\qquad S\in f(\cO_k),\quad T\in \cO_k.
\]
We say that the orbit-edges
\[
\{S,r(S)\}
\qquad\text{and}\qquad
\{T,r(T)\}
\]
are \emph{occupied} by $e_k$. By \Cref{lem:almost} and \Cref{lem:rotationorbits}, every edge of $F$ is occupied by at most one edge of $F'\setminus F$, see \Cref{fig:good-case}.

 Furthermore, for $j\in [m]$, call $\set{S, r(S)}$ \emph{$j$-active} if 
\[
f(\cO_k)\preceq \cO_j \prec \cO_k.
\]

We claim that for every $i\in [m]$ there exists a cycle
\[
C_i\subseteq \KG(\cT_n)
\]
spanning $\cO_1\cup \cdots \cup \cO_i$ and containing all $i$-active edges. The desired Hamiltonian cycle is then $C_m$.

The proof is by induction on $i$. For $i=1$, take the canonical cycle spanned by $\cO_1$.

Now let $i\ge 2$ and assume that there exists a cycle $C_{i-1}\subseteq \KG(\cT_n)$ spanning $\cO_1\cup \cdots \cup \cO_{i-1}$ and containing all $(i-1)$-active edges. Write
\[
e_i=\{S,T\},
\qquad S\in f(\cO_i),\quad T\in \cO_i.
\]
By definition, the edge $\{S,r(S)\}$ is $(i-1)$-active and thus belongs to $C_{i-1}$.
Construct $C_i$ by replacing the edge $\set{S,r(S)}$ in $C_{i-1}$ with the path
\[
S,\ T,\ r^{-1}(T),\ \dots,\ r^{-(\abs{\cO_i}-1)}(T)=r(T),\ r(S),
\]
which traverses the entire cycle on the orbit $\cO_i$ except for the occupied edge $\{T,r(T)\}$.
Clearly, $C_i$ spans $\cO_1\cup \cdots \cup \cO_i$.

It remains to verify that $C_i$ contains all $i$-active edges. The only $(i-1)$-active edge that ceases to be active at stage $i$ is precisely $\{S,r(S)\}$, which was replaced. Moreover, the edge $\{T,r(T)\}$ is occupied by $e_i$ itself, and hence is not $i$-active. All other edges on the cycle spanned by $\cO_i$ remain in $C_i$. Therefore $C_i$ contains all $i$-active edges, completing the induction.

At the end of the process, $C_m$ spans all rotation orbits and hence all of $\cT_n$. Therefore $C_m$ is a Hamiltonian cycle in $\KG(\cT_n)$.

\begin{figure}[htb!]
    \centering
    \includegraphics[width=0.6\linewidth]{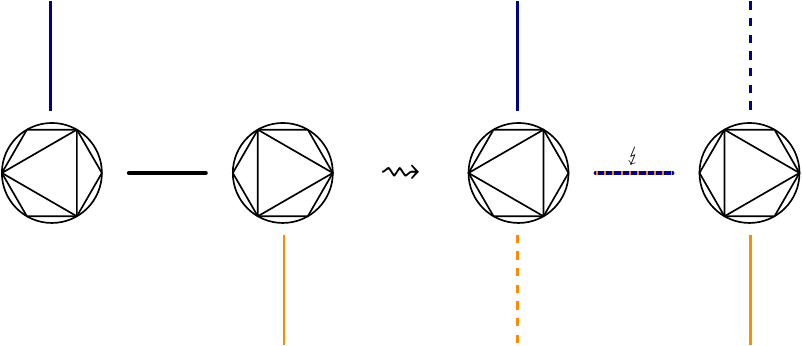}
    \caption{The orbit of ``all-ears'' triangulations of $n =6$ and why the orbit edge gets occupied more than once if each triangulation has degree two.}
    \label{fig:bad_case}
\end{figure}

Finally, for $n = 6$, let $\cO = \set{T_1, T_2}$ be the unique orbit of size two, i.e. where the triangulations only contain ears. In particular, $F \setminus \cO$ remains a $2$-factor. A careful reading of \Cref{lem:almost} gives that the degree increases at most by one when going from $F$ to $F'$, so $\deg_{F'}(T_1), \deg_{F'}(T_2) \leq 2$. Note that, as illustrated in \Cref{fig:bad_case}, the above construction only fails when $\deg_{F'}(T_1) = \deg_{F'}(T_2) = 2$ as the edge $\set{T_1, T_2}$ gets occupied by multiple edges. To avoid this, choose the starting vertex of $\GrayC{6}$ such that the unique\footnote{Recall that $\GrayC{n}$ is an independent set in $\KG\of{\cT_n}$.} representative of $\cO$ is the last element. Therefore, $\cO$ will never serve as a parent in $G$, so we get using a similar argument as in \cref{lem:degree} that $\deg_G(\cO) = 1$. In particular, \cref{lem:almost} now yields $F'$ such that $\min\set{\deg_{F'}(T_1),  \deg_{F'}(T_2)} = 1$, so $\set{T_1, T_2}$ gets only occupied once. Hence, for the remainder, we may proceed as for $n \geq 7$.
\end{proof}

\section{The permutohedron}\label{sec:perm}

In this section, we prove Theorem \ref{thm:permutohedron}. However, we shall begin with a warm-up proof that $\KG(\Perm_n)$ is Hamiltonian (the case $k=1$ of Theorem \ref{thm:permutohedron}), using the same general philosophy as in the proof of \Cref{thm:main}. First, recall that the facets of $\Perm_n$ are indexed by the nonempty proper subsets $\emptyset\neq S\subsetneq [n]$, where the facet $f_S$ consists of those permutations $\sigma\in S_n$ such that $\sigma(S)=[|S|]$. Hence two permutations $\sigma,\tau\in S_n$ are adjacent in $\KG(\Perm_n)$ if and only if there does not exist $k\in [n-1]$ such that $\sigma^{-1}([k])=\tau^{-1}([k])$.

 To show that $\KG(\Perm_n)$ is Hamiltonian, we shall again decompose the graph into subgraphs, and then use a Hamiltonian cycle in an auxiliary graph to glue a cycle in each subgraph together. The first observation is that in the case of $\KG(\Perm_n)$ the subgraphs can be chosen to be cliques. 

\begin{lemma}\label{lem:perm-cliques}
Let $\rho=(1\,2\,\cdots\,n)\in S_n$ and consider $H = \langle \rho \rangle$. The right cosets $H\sigma$ partition the vertices of $\KG(\Perm_n)$ into cliques.
\end{lemma}

\begin{proof}
Suppose, for contradiction, that there exist $\sigma\in S_n$ and $i\in [n-1]$ such that $\sigma$ and $\rho^i\sigma$ are not adjacent in $\KG(\Perm_n)$. Then there exists $k\in [n-1]$ such that $\sigma^{-1}([k])=(\rho^i\sigma)^{-1}([k])=\sigma^{-1}\rho^{-i}([k])$ and so $\rho^i([k]) = [k]$, a contradiction. 
\end{proof}

Let $\mathcal{C}_1,\dots,\mathcal{C}_N$ denote the right cosets of $H$, where $N=(n-1)!$.

\begin{lemma}\label{lem:perm-representative}
For each right coset $\mathcal{C}$, there exists a unique permutation $\pi\in \mathcal{C}$ such that $\pi(n)=n$.
\end{lemma}

\begin{proof}
Fix $\sigma\in \mathcal{C}$ and let $i=\sigma(n)$. The elements of $\mathcal{C}$ are $\sigma,\rho\sigma,\rho^2\sigma,\dots,\rho^{n-1}\sigma$. Hence $\rho^{n-i}\sigma$ is the unique element of $\mathcal{C}$ satisfying $\rho^{n-i}\sigma(n)=\rho^{n-i}(i)=n$. 
\end{proof}

We call the unique permutation $\pi\in\mathcal{C}$ with $\pi(n)=n$ the \emph{marking permutation} of $\mathcal{C}$. The set of marking permutations is naturally identified with $S_{n-1}$. Let $\pi_1,\pi_2,\dots,\pi_N$
be a Hamiltonian cycle in the $1$-skeleton of $\Perm_{n-1}$. Such a cycle exists by the Steinhaus-Johnson-Trotter algorithm \cite{Steinhaus1964, Johnson1963, Trotter1962}. 

For each $i\in [N]$, let $\mathcal{C}_i$ be the right coset containing $\pi_i$. Since $\pi_i(n)=n$, the permutation $\sigma_i\coloneqq \rho\pi_i$
lies in the same clique $\mathcal{C}_i$ and satisfies $\sigma_i(n)=1.$
We claim that $\sigma_i$ is adjacent to $\pi_{i+1}$. Indeed, for every $k\in [n-1]$ we have $n\in \sigma_i^{-1}([k])$, because $\sigma_i(n)=1\in [k]$.
On the other hand, $n\notin \pi_{i+1}^{-1}([k])$, because $\pi_{i+1}(n)=n\notin [k]$. Thus $\sigma_i^{-1}([k])\neq \pi_{i+1}^{-1}([k])$ for all $k\in [n-1]$ and therefore $\sigma_i$ and $\pi_{i+1}$ are adjacent in $\KG(\Perm_n)$.

Now enumerate each clique $\mathcal{C}_i$ by a path $\alpha_i$ that starts at $\pi_i$ and ends at $\sigma_i$. Since $\mathcal{C}_i$ is a clique, such a path spanning $\mathcal{C}_i$ exists. Concatenating $\alpha_1,\alpha_2,\dots,\alpha_N$ yields a Hamiltonian cycle in $\KG(\Perm_n)$. \qed

\smallskip

The argument above is very much in the spirit of \cite{ChenFuredi2002, BellmannSchuelke2021}: we start with a clique factor and then glue the cliques $\mathcal{C}_i$ using the Hamiltonian cycle (in the $1$-skeleton of $\Perm_{n}$) on the marking permutations. 
However, this method seems to only show Hamiltonicity. To obtain the $k$th power of a Hamiltonian cycle, one would need much more than a single bridge between consecutive cliques: one would need a whole family of compatible bridges guaranteeing that the last $k$ vertices of one clique are adjacent to the first $k$ vertices of the next. We do not see a natural way to force such a structure from the clique-splicing argument. Instead, the higher-power statement follows much more naturally from density considerations.

A permutation $\sigma\in S_n$ is \emph{decomposable} if there exists $p\in [n-1]$ such that $\sigma([p])=[p]$, and \emph{indecomposable} otherwise. Let $\mathcal{I}_n$ denote the set of indecomposable permutations. The \emph{Cayley graph} of $S_n$ generated by a set of permutations $I\subset S_n$ is the graph $\Cay(S_n,I)$ whose vertices are the elements of $S_n$ and where $\sigma$ and $\tau$ form an edge if and only if $\tau^{-1}\sigma \in I$. We record a standard observation that $\KG(\Perm_n)$ is the Cayley graph of $S_n$ generated by the indecomposable permutations. 

\begin{lemma}\label{lem:cayley}
\[
\KG(\Perm_n)=\Cay(S_n,\mathcal{I}_n).
\]
\end{lemma}

\begin{proof}
By the facet description above, two permutations $\sigma,\tau\in S_n$ are non-adjacent in $\KG(\Perm_n)$ if and only if $\sigma^{-1}([p])=\tau^{-1}([p])$ for some $p\in [n-1]$. Equivalently, $(\tau^{-1}\sigma)([p])=[p]$, that is, $\tau^{-1}\sigma$ is decomposable.
Hence $\sigma$ and $\tau$ are adjacent if and only if $\tau^{-1}\sigma$ is indecomposable.
\end{proof}

We record the following standard estimate for the number of indecomposable permutations.

\begin{lemma}\label{lem:indecomp}
We have
\[
|\mathcal{I}_n|
=
\left(1-O\!\left(\frac{1}{n}\right)\right)n!.
\]
\end{lemma}

In fact, the following more precse inequality holds:
\[
n!-|\mathcal{I}_n|
\le
n!\sum_{p=1}^{n-1}\binom{n}{p}^{-1}
\]
See for example \cite{Comtet1972}. Using this, we can now easily prove \Cref{thm:permutohedron}.

\begin{proof}[Proof of \Cref{thm:permutohedron}]
By \Cref{lem:cayley}, the graph $\KG(\Perm_n)$ is the Cayley graph $\Cay(S_n,\mathcal{I}_n)$, so every vertex has degree $|\mathcal{I}_n|$. By \Cref{lem:indecomp},
\[
\delta\!\bigl(\KG(\Perm_n)\bigr)
=
|\mathcal{I}_n|
=
\left(1-O\!\left(\frac{1}{n}\right)\right)n!.
\]
Fix $k\ge 1$. For all sufficiently large $n$, this minimum degree is larger than $\frac{k}{k+1}\,n!$.
The theorem of Koml\'os, S\'ark\"ozy, and Szemer\'edi \cite{KomlosSarkozySzemeredi1998} therefore implies that $\KG(\Perm_n)$ contains the $k$th power of a Hamiltonian cycle.
\end{proof}

\section{Concluding remarks}

The proof of \Cref{thm:main} and the work by Merino-M\"utze-Namrata from \cite{MerinoMutzeNamrata2023} together suggest a general paradigm for establishing Hamiltonicity in Kneser graphs $\KG(\mathcal{F})$ of various set systems $\mathcal{F}$. The key ingredients are:
\begin{itemize}[leftmargin=1.5em]
    \item a decomposition of the graph into cycles arising from symmetry,
    \item a sparse guide cycle in an auxiliary graph on $\mathcal{F}$ that meets every component of the decomposition,
    \item and a local bridge phenomenon allowing one to convert transitions in the guide cycle into edges of the target graph.
\end{itemize}
A compatible choice of such bridges then allows one to splice all cycles into a single Hamiltonian cycle. In fact, in our proof for \Cref{thm:main}, a Hamiltonian path in the auxiliary graph would already be sufficient; the cyclic formulation is simply the most natural one.

It is natural to ask how far this paradigm extends. For example, a particularly appealing setting is that of graph associahedra \cite{CarrDevadoss2006,Dev09}. For any graph $G$, the $G$-associahedron $\Assoc(G)$ is a simple polytope whose vertices correspond to maximal tubings on $G$. When $G$ is a path, this is the standard associahedron, and when $G$ is the complete graph this recovers the permutohedron. In a recent paper, Manneville and Pilaud \cite{MannevillePilaud2016} showed that the $1$-skeleton of $\Assoc(G)$ is Hamiltonian for every graph $G$ with at least two edges. 
This naturally leads to the following question.

\begin{question}
    For which graphs $G$ is the Kneser graph $\KG(\Assoc(G))$ Hamiltonian?
\end{question}

\medskip

Theorem~\ref{thm:main} resolves this question when $G$ is a path, while \Cref{sec:perm} addresses the complete graph case. We believe that the strategy developed here should extend to a much broader class of graphs. For example, when $G$ is a cycle, the corresponding polytope is the cyclohedron, whose vertices admit a natural interpretation in terms of centrally symmetric triangulations of a convex polygon. In this setting, the Hamiltonian cycle of the $1$-skeleton translates to such a cycle for the flip graph\footnote{Though more than one diagonal may get flipped to preserve being centrally symmetric.}, and the same orbit-decomposition and bridge argument appears to carry over with only minor modifications.

At the opposite extreme, when $G$ is dense, one can often establish Hamiltonicity of $\KG(\Assoc(G))$ by much simpler means. Indeed, for sufficiently dense graphs $G$, the resulting Kneser graph is itself very dense, and Dirac-type arguments imply Hamiltonicity directly. For instance, when $G=K_n$ one recovers the permutohedron, and the Hamiltonicity of $\KG(\Perm_n)$ follows from a straightforward density argument. More generally, if $G$ is obtained from $K_n$ by deleting a sparse set of edges (for example, a perfect matching), then $\KG(\Assoc(G))$ still has very large minimum degree, and the same approach applies.

These two extremes, highly structured sparse graphs such as paths and cycles, and highly connected dense graphs, suggest that Hamiltonicity of $\KG(\Assoc(G))$ might hold in considerable generality. It would be interesting to determine whether a unified argument exists, or whether fundamentally different mechanisms are required in the sparse and dense regimes.

\section*{Acknowledgments}
We thank Daniel Zhu for useful discussions in the initial stages of this project. C.P.\ was supported by NSF grant DMS-2246659.


\begin{thebibliography}{16}
\bibitem{BellmannSchuelke2021}
J. Bellmann and B. Sch\"ulke,
\emph{Short proof that Kneser graphs are Hamiltonian for $n\geq4k$}, Discrete Math. {\bf 344} (2021), no.~7, Paper No. 112430, 2 pp.; MR4247012

\bibitem{CarrDevadoss2006}
M.~P. Carr and S.~L. Devadoss,
\emph{Coxeter complexes and graph-associahedra}, Topology Appl. {\bf 153} (2006), no.~12, 2155--2168; MR2239078

\bibitem{ChenLih1987}
B.-L. Chen and K.-W. Lih,
\emph{Hamiltonian uniform subset graphs}, J. Combin. Theory Ser. B {\bf 42} (1987), no.~3, 257--263; MR0888679


\bibitem{Chen2000}
Y.-C. Chen, 
\emph{Kneser graphs are Hamiltonian for $n\geq 3k$}, J. Combin. Theory Ser. B {\bf 80} (2000), no.~1, 69--79; MR1778200

\bibitem{Chen2003}
Y.-C. Chen, 
\emph{Triangle-free Hamiltonian Kneser graphs}, J. Combin. Theory Ser. B {\bf 89} (2003), no.~1, 1--16; MR1999733

\bibitem{ChenFuredi2002}
Y.-C. Chen and Z. F\"uredi, 
\emph{Hamiltonian Kneser graphs}, Combinatorica {\bf 22} (2002), no.~1, 147--149; MR1883565

\bibitem{ChristofidesHladkyMathe2014}
D. Christofides, J. Hladk\'y{} and A. M\'ath\'e, \emph{Hamilton cycles in dense vertex-transitive graphs}, J. Combin. Theory Ser. B {\bf 109} (2014), 34--72; MR3269902

\bibitem{Comtet1972}
L. Comtet,
\emph{Sur les coefficients de l'inverse de la s\'erie formelle $\sum n!t\sp{n}$}, C. R. Acad. Sci. Paris S\'er. A-B {\bf 275} (1972), {\rm A}569--{\rm A}572; MR0302457

\bibitem{Dev09} 
S.~L. Devadoss, \emph{A realization of graph associahedra}, Discrete Math. {\bf 309} (2009), no.~1, 271--276; MR2479448

\bibitem{EKR1961}
P. Erd\H os, C. Ko and R. Rado, 
\emph{Intersection theorems for systems of finite sets}, Quart. J. Math. Oxford Ser. (2) {\bf 12} (1961), 313--320; MR0140419


\bibitem{HeinrichWallis1978}
K. Heinrich and W.~D. Wallis,
\emph{Hamiltonian cycles in certain graphs}, J. Austral. Math. Soc. Ser. A {\bf 26} (1978), no.~1, 89--98; MR0510592

\bibitem{HurtadoNoy1999}
F. Hurtado and M. Noy, 
\emph{Graph of triangulations of a convex polygon and tree of triangulations}, Comput. Geom. {\bf 13} (1999), no.~3, 179--188; MR1723053

\bibitem{Johnson1963}
S.~M. Johnson, \emph{Generation of permutations by adjacent transposition}, Math. Comp. {\bf 17} (1963), 282--285; MR0159764

\bibitem{King2006}
A.~D. King, \emph{Generating indecomposable permutations}, Discrete Math. {\bf 306} (2006), no.~5, 508--518; MR2212519

\bibitem{KomlosSarkozySzemeredi1998}
J. Koml\'os, G.~N. S\'ark\"ozy and E. Szemer\'edi, \emph{Proof of the Seymour conjecture for large graphs}, Ann. Comb. {\bf 2} (1998), no.~1, 43--60; MR1682919

\bibitem{Lovasz1970}
L. Lov\'asz,
\emph{Problem 11},
in Combinatorial Structures and their Applications (1970).

\bibitem{Lovasz1978}
L. Lov\'asz, \emph{Kneser's conjecture, chromatic number, and homotopy}, J. Combin. Theory Ser. A {\bf 25} (1978), no.~3, 319--324; MR0514625

\bibitem{Lucas1987}
J.~M. Lucas, \emph{The rotation graph of binary trees is Hamiltonian}, J. Algorithms {\bf 8} (1987), no.~4, 503--535; MR0920505

\bibitem{MannevillePilaud2016}
T. Manneville and V. Pilaud,
\emph{Graph properties of graph associahedra}, S\'em. Lothar. Combin. 73 ([2014--2016]), Art. B73d, 31 pp.; MR3383157

\bibitem{MerinoMutzeNamrata2023}
A.~I. Merino, T. M\"utze and Namrata, \emph{Kneser graphs are Hamiltonian}, Adv. Math. {\bf 468} (2025), Paper No. 110189, 80 pp.; MR4876460


\bibitem{MutzeNummenpaloWalczak2021}
T. M\"utze, J. Nummenpalo and B. Walczak, \emph{Sparse Kneser graphs are Hamiltonian}, J. Lond. Math. Soc. (2) {\bf 103} (2021), no.~4, 1253--1275; MR4273468


\bibitem{MolnarPohoataZhengZhu2025}
A. Molnar, C. Pohoata, M. Zheng, and D. Zhu,
\emph{A Lov\'asz--Kneser theorem for triangulations},
arXiv:2510.27689.

\bibitem{Postnikov2009}
A. Postnikov, \emph{Permutohedra, associahedra, and beyond}, Int. Math. Res. Not. IMRN {\bf 2009}, no.~6, 1026--1106; MR2487491

\bibitem{Steinhaus1964}
H. Steinhaus, {\it One hundred problems in elementary mathematics}, Basic Books, New York, 1964; MR0157881

\bibitem{Trotter1962}
H. F. Trotter. \emph{Algorithm 115}: Perm. Commun. ACM, 5(8):434–435, 1962

\end{thebibliography}
\end{document}